\documentclass[12pt]{amsart}
\usepackage{amsmath}
\usepackage{amsfonts}

\DeclareMathOperator*{\W}{Wig}

\newtheorem{Th}{Theorem}[section]

\newtheorem{Lemma}[Th]{Lemma}
\newtheorem{Def}[Th]{Definition}
\newtheorem{Prop}[Th]{Proposition}

\title{Compactness of localization operators on modulation spaces of $\omega$-tempered distributions}

 \author{Chiara Boiti \and Antonino De Martino}

\begin{document}

\maketitle
\markboth{Compactness of localization operators on modulation spaces...}{Chiara Boiti and Antonino De Martino}

\begin{abstract}
We give sufficient conditions for compactness of localization operators on modulation spaces $ \textbf{M}^{p,q}_{m_{\lambda}}( \mathbb{R}^{d})$ of
$\omega$-tempered distributions whose short-time Fourier transform is in the weighted
mixed space $L^{p,q}_{m_\lambda}$ for  $m_\lambda(x)=e^{\lambda\omega(x)}$.
\end{abstract}

\section{Introduction and main results}
In this paper we study some properties of localization operators, which are 
pseudo-differential operators of time-frequency analysis suitable for 
applications to the reconstruction of signals, because they allow to recover a 
filtered version of the original signal. To introduce the problem, let us recall
the \emph{translation} and \emph{modulation} operators
\begin{eqnarray*}
T_{x}f(y)=f(y-x), \quad M_{\xi}f(y)=e^{i y \cdot \xi}f(y), \qquad  x,y \in \mathbb{R}^{d},
\end{eqnarray*}
and, for a window function $ \psi \in L^{2}(\mathbb{R}^{d})$, the \emph {short-time Fourier transform} (briefly STFT) of a function $ f \in L^{2}(\mathbb{R}^{d})$
$$ 
V_{\psi} f(z)= \langle f, M_{\xi}T_{x}\psi \rangle =\int_{\mathbb{R}^{d}}f(y) \overline{\psi(y-x)}e^{-iy \cdot \xi} \, dy,  \qquad z=(x,\xi) \in \mathbb{R}^{2d}.
$$
With respect to the inversion formula for the STFT (see \cite[Cor.~3.2.3]{G})
$$
f= \frac{1}{(2 \pi)^{d} \langle \gamma, \psi \rangle} \int_{\mathbb{R}^{2d}} V_{\psi}f(x, \xi)M_{\xi}T_{x} \gamma\,  dx d \xi,
$$
which gives a reconstruction of the signal $f$, the 
localization operator, as defined in \eqref{local}, modifies $V_{\psi}f(x,\xi)$ by multiplying it by a suitable $a(x,\xi)$ before reconstructing the signal, so that a filtered version of the original signal $f$ is recovered.

Another important operator in time-frequency analysis that we shall need in the following is the \emph{cross-Wigner transform} defined, for $f,g \in L^{2}(\mathbb{R}^{d})$, by
$$ 
\W(f,g)(x, \xi)= \int_{\mathbb{R}^{d}} f \bigl(x+ \frac{t}{2} \bigl) \overline{g \bigl(x- \frac{t}{2} \bigl)}  e^{-i  \xi\cdot t } \, dt \qquad x, \xi \in \mathbb{R}^{d}.
$$
The {\em Wigner transform} of $f$ is then defined by $\W f:= \W(f,f)$.

The above Fourier integral operators,
with standard generalizations to more general spaces of functions or distributions,
 have been largely investigated in time-frequency analysis. In particular, results about boundedness or compactness related to the subject of this paper can be
found, for instance, in \cite{BGHO,CG,FG,FG2,FGP,P,T}.

Inspired by \cite{CG,FG}, our aim in this paper is to study boundedness of localization operators on modulation spaces in the setting of $\omega$-tempered distributions,
 for a weight functions $ \omega$ defined as below:
\begin{Def}
\label{peso}
A \emph{non-quasianalytic subadditive weight function} is a continous increasing function 
$ \omega: [0, + \infty) \to [0, + \infty)$ satisfying the following properties:
\end{Def}
\begin{itemize}
\item[($ \alpha$)]
 $ \ \ \omega(t_{1}+t_{2}) \leq \omega(t_{1})+ \omega(t_{2}), \qquad \forall t_{1}, t_{2} \geq 0;$
\item[($ \beta $)]
 $ \ \  \int_{1}^{+\infty} \frac{\omega(t)}{t^{2}} \, dt < + \infty ;$
\item[($ \gamma$)]
 $ \ \ \exists A \in \mathbb{R}$, $B>0$ s.t $ \omega(t) \geq A+B \log(1+t), \qquad \forall t \geq 0;$
\item[($\delta$)]
 $ \ \ \varphi_{\omega}(t):= \omega(e^{t})$ is convex.
\end{itemize}
We then consider $ \omega(\xi):= \omega(| \xi|)$ for $ \xi \in \mathbb{C}^{d}$.

\begin{Def}
\label{SW2}
The space $ \mathcal{S}_{\omega}(\mathbb{R}^{d})$ is defined as the set of all $ u \in L^{1}(\mathbb{R}^{d})$ such that $ u, \hat{u} \in C^{\infty}(\mathbb{R}^{d})$  and
\end{Def}
\begin{itemize}
\item[(i)] 
$ \ \forall \lambda>0, \alpha \in \mathbb{N}^{d}_{0}$: $ \sup_{x \in \mathbb{R}^{d}} e^{\lambda \omega(x)} |D^{\alpha}u(x)| < + \infty ,$
\item[(ii)] 
$\ \ \forall \lambda>0, \alpha \in \mathbb{N}^{d}_{0}$: $ \sup_{\xi \in \mathbb{R}^{d}} e^{\lambda \omega(\xi)} |D^{\alpha}\hat{u}(\xi)| < + \infty,$
\end{itemize}
where $ \mathbb{N}_{0}:= \mathbb{N} \cup \{ 0 \}$.

\vspace{2mm}
Note that for $\omega(t)= \log(1+t)$ we obtain the classical Schwartz class $ \mathcal{S}(\mathbb{R}^{d})$, while in general $ \mathcal{S}_{\omega}(\mathbb{R}^{d}) \subseteq \mathcal{S}(\mathbb{R}^{d})$.
For more details about the spaces $ \mathcal{S}_{\omega}(\mathbb{R}^{d})$ we refer to \cite{BJO1}-\cite{BJOS}. In particular, we can define on $ \mathcal{S}_{\omega}(\mathbb{R}^{d})$ different equivalent systems of seminorms that make 
$ \mathcal{S}_{\omega}(\mathbb{R}^{d})$ a 
Fr\'echet nuclear space. It is also an algebra under multiplication and convolution.

The corresponding strong dual space is denoted by 
$ \mathcal{S}'_{\omega}(\mathbb{R}^{d})$ and its elements are called {\em $\omega$-tempered distributions}. Moreover, $ \mathcal{S}' (\mathbb{R}^{d})\subseteq \mathcal{S}'_{\omega}(\mathbb{R}^{d})$ and the Fourier Transform, the short-time Fourier transform and the Wigner transform are continous from 
$ \mathcal{S}_{\omega}(\mathbb{R}^{d})$ to $ \mathcal{S}_{\omega}(\mathbb{R}^{d})$ and from $ \mathcal{S}'_{\omega}(\mathbb{R}^{d})$ to 
$ \mathcal{S}'_{\omega}(\mathbb{R}^{d})$.

 The "right" function spaces in time-frequency analysis to work with the STFT are the so-called {\em modulation spaces}, introduced by H.~Feichtinger in \cite{F}. In this context, 
 we consider the weight $m_\lambda(z):=e^{\lambda\omega(z)}$, for $\lambda\in\mathbb R$, and
 define $L^{p,q}_{m_{\lambda}}(\mathbb{R}^{2d})$ as the space of measurable functions $f$ 
 on $\mathbb R^{2d}$ such that
 $$
 \| f \|_{L^{p,q}_{m_{\lambda}}}:= \int_{\mathbb{R}^{d}} \biggl( \int_{\mathbb{R}^{d}} |f(x, \xi)|^{p} m_{\lambda}(x, \xi)^{p} \, dx \bigl)^{\frac{q}{p}} \, d \xi \biggl)^{\frac{1}{q}}< + \infty,
 $$
for $ 1 \leq p,q < + \infty$, with standard changes if $p$ (or $q$) is $ + \infty$. We define then, for $1 \leq p,q \leq + \infty$, the modulation space
$$ 
\textbf{M}_{m_{\lambda}}^{p,q}(\mathbb{R}^{d}):= \{ f \in \mathcal{S}'_{\omega}(\mathbb{R}^{d}): V_{\varphi}f \in L^{p,q}_{m_{\lambda}}(\mathbb{R}^{2d}) \},
$$
which is independent of the window function 
$ \varphi \in \mathcal{S}_{\omega}(\mathbb{R}^{d}) \setminus \{ 0 \}$ and is a Banach 
space with norm 
$ \| f \|_{\textbf{M}_{m_{\lambda}}^{\,p,q}}:= \| V_{\varphi} f \|_{L^{p,q}_{m_{\lambda}}}$ 
(see \cite{BJO2}). Moreover, for $ 1 \leq p,q <+ \infty$, the space 
$ \mathcal{S}_{\omega}(\mathbb{R}^{d})$ is a dense subspace of 
$ \textbf{M}_{m_{\lambda}}^{p,q}$ by \cite[Prop. 3.9]{BJO2}. We shall denote 
$\textbf{M}_{m_{\lambda}}^{p}(\mathbb{R}^{d})= 
\textbf{M}_{m_{\lambda}}^{p,p}(\mathbb{R}^{d})$ and 
$ \textbf{M}^{p,q}(\mathbb{R}^{d})=\textbf{M}_{m_0}^{p,q}(\mathbb{R}^{d})$.

As in \cite[Thm. 12.2.2]{G} if $ p_{1} \leq p_{2}$, $ q_{1} \leq q_{2}$, and 
$ \lambda \leq \mu$ then 
$ \textbf{M}^{\, p_{1}, q_{1}}_{m_{\mu}} \subseteq \textbf{M}^{\, p_{2}, q_{2}}_{m_{\lambda}}$
with continous inclusion (see \cite[Lemma 2.3.16]{D}).
Set
\begin{eqnarray*}
&& m_{\lambda, 1}(x):= m_{\lambda}(x,0), \quad m_{\lambda, 2}(x):= m_{\lambda}(0,\xi),\\
&&v_{\lambda}(z)=e^{| \lambda| \omega(z)}, \quad v_{\lambda, 1}(x):= v_{\lambda}(x,0), \quad v_{\lambda, 2}(x):= v_{\lambda}(0,\xi),
\end{eqnarray*}
and prove the following generalization of \cite[Prop. 2.4]{CG}:
\begin{Prop}
\label{24Gb}
Let $1 \leq p,q,r,t, t' \leq + \infty$ such that
$ \frac{1}{p}+ \frac{1}{q}-1= \frac{1}{r}$ and $\frac{1}{t}+ \frac{1}{t'}=1$.
Then, for all $ \lambda, \mu \in \mathbb{R}$ and $ 1 \leq s \leq + \infty$,
\begin{eqnarray}
\nonumber
\textbf{M}^{ \, p,st}_{m_{\lambda, 1} \otimes m_{\mu, 2}}( \mathbb{R}^{d}) * \textbf{M}^{\, q, st'}_{m_{\lambda,1} \otimes v_{\lambda,2} m_{-\mu,2}}(\mathbb{R}^{d}) \hookrightarrow \textbf{M}_{m_{\lambda}}^{ \, r,s}(\mathbb{R}^{d})
\end{eqnarray}
\begin{eqnarray}
\label{2632}
 \hspace*{-20mm}\mbox{and}\qquad\qquad \| f * g \|_{\textbf{M}_{m_{\lambda}}^{r,s}} \leq \| f \|_{\textbf{M}^{\, p,st}_{m_{\lambda, 1} \otimes m_{\mu, 2}}}  
\| g \|_{\textbf{M}^{\, q, st'}_{m_{\lambda,1} \otimes v_{\lambda,2} m_{-\mu,2}}}.
\end{eqnarray}
\end{Prop}

\begin{proof}
For the Gaussian function $g_{0}(x)=e^{- \pi |x|^{2}} \in \mathcal{S}_{\omega}(\mathbb{R}^{d})$ consider on $ \textbf{M}^{r,s}_{m_{\lambda}}$ the modulation norm with respect to the window function $ g(x):=g_{0}* g_{0}(x)= 2^{-d/2} e^{- \frac{\pi}{2} | x |^{2}} \in \mathcal{S}_{\omega}(\mathbb{R}^{d})$.
Since $m_{\lambda}(x, \xi) \leq m_{\lambda}(x,0) v_{\lambda} (0, \xi)$ and $\overline{g_{0}(-x)}=g_{0}(x)$, by \cite[Lemma 3.1.1]{G}, Young and H\"older inequalities:
\begin{eqnarray*}
&& \| f *h \|_{\textbf{M}^{r,s}_{m_{\lambda}}}  = \| V_{g}(f*h) \|_{L^{r,s}_{m_{\lambda}}}= \biggl( \int_{\mathbb{R}^{d}} \biggl( \int_{\mathbb{R}^{d}} | V_{g}(f*h)|^{r} m_{\lambda}^{r} (x, \xi) \, dx \biggl)^{\frac{s}{r}} d \xi \biggl)^{\frac{1}{s}} \\
&& \leq \biggl( \int_{\mathbb{R}^{d}} \biggl( \int_{\mathbb{R}^{d}} | (f*M_{\xi}g_{0})*(h*M_{\xi}g_{0})(x)|^{r} m_{\lambda}(x,0)^{r} \, dx \biggl)^{\frac{s}{r}} v_{\lambda}^{s}(0, \xi) \, d \xi \biggl)^{\frac{1}{s}}\\
&& = \biggl( \int_{\mathbb{R}^{d}} \| (f*M_{\xi}g_{0})*(h*M_{\xi}g_{0}) \|^{s}_{L^{r}_{m_{\lambda,1}}} v_{\lambda}^{s}(0, \xi) \, d \xi  \biggl)^{\frac{1}{s}}\\
&& \leq \biggl( \int_{\mathbb{R}^{d}} \| f * M_{\xi}g_{0} \|_{L^{p}_{m_{\lambda,1}}}^{s} \| h * M_{\xi}g_{0} \|_{L^{q}_{m_{\lambda,1}}}^{s} v_{\lambda}^{s}(0, \xi) \, d \xi  \biggl)^{\frac{1}{s}}\\
&& = \biggl( \int_{\mathbb{R}^{d}} \| V_{g_{0}}f \|_{L^{p}_{m_{\lambda,1}}}^{s}  m_{\mu}^{s}(0, \xi) \| V_{g_{0}}h \|_{L^{q}_{m_{\lambda,1}}}^{s} m_{- \mu}^{s}(0, \xi) v_{\lambda}^{s}(0, \xi) \, d \xi  \biggl)^{\frac{1}{s}}\\
&& \leq \| f \|_{\textbf{M}^{p,st}_{m_{\lambda,1}\otimes m_{\mu,2}}} \| h \|_{\textbf{M}^{q, st'}_{m_{\lambda,1} \otimes v_{\lambda,2} m_{- \mu,2}}}.
\end{eqnarray*} 
\end{proof}

Given two window functions $ \psi, \gamma \in \mathcal{S}_{\omega}(\mathbb{R}^{d}) \setminus \{ 0 \}$ and a symbol $ a \in \mathcal{S}'_{\omega}(\mathbb{R}^{2d})$, the corresponding {\em localization operator} $L^{a}_{\psi, \gamma}$ is defined, for 
$f \in \mathcal{S}_{\omega}(\mathbb{R}^{d})$, by
\begin{equation}
\label{local}
L^{a}_{\psi, \gamma}f=V^{*}_{\gamma}(a \cdot V_{\psi} f)= \int_{\mathbb{R}^{2d}} a(x, \xi) V_{\psi} f(x, \xi) M_{\xi}T_{x} \gamma \,dx d \xi,
\end{equation}
where $V^{*}_{\gamma}$ is the adjoint of $V_{\gamma}$.
As in \cite[Lemma 2.4]{BCG} we have that $L^{a}_{\psi, \gamma}$ is a Weyl operator $L^{a^{w}}$ with symbol $a^{w}=a * \hbox{Wig}( \gamma, \psi)$:
\begin{equation}
\label{weyl}
L^{a^{w}}f:= \frac{1}{(2 \pi)^{d}} \int_{\mathbb{R}^{2d}} \hat{a}^{w}(\xi, u) 
e^{-i \xi \cdot u} T_{-u}M_{\xi}f \,du d \xi.
\end{equation}
Moreover, if $ f,g \in \mathcal{S}_{\omega}(\mathbb{R}^{d})$ then by definition of adjoint operator we can write
\begin{equation}
\nonumber
\langle L^{a}_{\psi, \gamma}f,g \rangle = \langle a \cdot V_{\psi}f, V_{\gamma}g \rangle =\langle a , \overline{ V_{\psi}f} V_{\gamma}g \rangle,
\end{equation}
and, similarly as in \cite[Thm. 14.5.2]{G} (see also \cite[Teo. 2.3.21]{D}), we have, for $a^{w} \in \textbf{M}^{\infty,1}_{m_{\mu}}( \mathbb{R}^{2d})$ with $ \mu \geq 0$, 
\begin{equation}
\label{in2}
 \| L^{a^{w}} f \|_{\textbf{M}^{p,q}_{m_{\lambda}}}= \| L^{a}_{\psi, \gamma} f \|_{\textbf{M}^{p,q}_{m_{\lambda}}} \leq \| a^{w} \|_{\textbf{M}^{\infty,1}_{m_{\mu}}} \| f \|_{\textbf{M}^{p,q}_{m_{\lambda}}},
\end{equation}
for all $ f \in \textbf{M}^{p,q}_{m_{\lambda}}$ and $ \lambda \in \mathbb{R}$.

\begin{Th}
\label{lim1}
Let $ \psi, \gamma \in \mathcal{S}_{\omega}(\mathbb{R}^{d}) \setminus \{0\}$ and $ a \in \textbf{M}^{\infty}_{m_{\lambda}}( \mathbb{R}^{2d})$ for some $ \lambda \geq 0$. Then $L^{a}_{\psi, \gamma}$ is bounded from $ \textbf{M}^{p,q}_{m_{\lambda}}(\mathbb{R}^{d})$ to  $ \textbf{M}^{p,q}_{m_{\lambda}}(\mathbb{R}^{d})$, for $1 \leq p,q < + \infty$, and
\begin{eqnarray*}
\| L^{a}_{\psi, \gamma} \|_{op}  \leq \| a \|_{\textbf{M}^{\infty}_{m_{- \lambda, 2}}} \| \psi \|_{\textbf{M}^{1}_{v_{\lambda}}} \| \gamma \|_{\textbf{M}^{p}_{m_{\lambda}}}.
\end{eqnarray*}
\end{Th}

\begin{proof}
By definition $V_{\psi}: \textbf{M}^{p,q}_{m_{\lambda}} \to 
L^{p,q}_{m_{\lambda}}( \mathbb{R}^{2d})$ and, by 
\cite[Prop. 3.7]{BJO2}, $V_{\gamma}^{*}: L^{p,q}_{m_{\lambda}}( \mathbb{R}^{2d}) \to \textbf{M}^{p,q}_{m_{\lambda}} (\mathbb{R}^{d})$. 
Let $f \in \textbf{M}^{p,q}_{m_{\lambda}}( \mathbb{R}^{d})$. To prove that 
$L^{a}_{\psi, \gamma}f=V^{*}_{\gamma}(a \cdot V_{\psi} f) \in 
\textbf{M}^{p,q}_{m_{\lambda}}$, it is then enough to show that 
$a \cdot V_{\psi} f \in L^{p,q}_{m_{\lambda}}(\mathbb{R}^{2d})$. 
By the inversion formula \cite[Prop. 3.7]{BJO2}, given two window functions 
$\Phi, \Psi \in \mathcal{S}_{\omega}(\mathbb{R}^{2d})$ with $ \langle \Phi, \Psi \rangle \neq 0$, 
we have, for $z=(z_1,z_2)\in\mathbb R^{2d}\times\mathbb R^{2d}$,
\begin{eqnarray*}
&& \biggl( \int_{\mathbb{R}^{d}} \biggl( \int_{\mathbb{R}^{d}} |a(x, \xi)|^{p} |V_{\psi}f(x, \xi)|^{p} e^{p \lambda \omega(x, \xi)} \, dx \biggl)^{\frac{q}{p}} \, d \xi \biggl)^{\frac{1}{q}}\\
& \leq& \frac{1}{(2 \pi)^{d}} \frac{1}{|\langle \Phi, \Psi \rangle| } \biggl(\int_{\mathbb{R}^{d}} \biggl( \int_{\mathbb{R}^{d}} \biggl( \int_{\mathbb{R}^{4d}} |V_{\Psi}a(z)|^{p}| M_{z_2}T_{z_1} \Phi(x, \xi)|^{p}  dz \biggl)\\
&& \cdot |V_{\psi}(x, \xi)|^{p} e^{p \lambda \omega(x, \xi)}  dx \biggl)^{\frac{q}{p}} d \xi \biggl)^{\frac{1}{q}}\\
& \leq & \frac{1}{(2 \pi)^{d}} \frac{1}{|\langle \Phi, \Psi \rangle| } \biggl(\int_{\mathbb{R}^{d}} \!  \biggl( \!  \int_{\mathbb{R}^{d}} \biggl(\!  \int_{\mathbb{R}^{4d}} \bigl( |V_{\Psi}a(z)| e^{\lambda \omega(z)}\bigl)^{p}| M_{z_2}T_{z_1} \Phi(x, \xi)|^{p}  dz \biggl)\\
&& \cdot |V_{\psi}(x, \xi)|^{p} e^{p \lambda \omega(x, \xi)}  dx \biggl)^{\frac{q}{p}} \!   d \xi \biggl)^{\frac{1}{q}}
\end{eqnarray*}
\begin{eqnarray*}
& \leq & C \| V_{\Psi}a \|_{L^{\infty}_{m_{\lambda}}}  \cdot \| V_{\psi} 
f \|_{L^{p,q}_{m_{\lambda}}}=C \| a \|_{\textbf{M}^{\infty}_{m_{\lambda}}} 
\cdot \| f \|_{\textbf{M}^{p,q}_{m_{\lambda}}},
\end{eqnarray*}
for some $C>0$.
Therefore $a \cdot V_{\psi} f \in L^{p,q}_{m_{\lambda}}(\mathbb{R}^{2d})$ and $L^{a}_{\psi,\gamma} f \in \textbf{M}^{p,q}_{m_{\lambda}}(\mathbb{R}^{d})$. 

To prove that $L^{a}_{\psi. \gamma}$ is bounded, consider $g \in  \mathcal{S}_{\omega}(\mathbb{R}^{d})$ and set $ \Psi=\W(g,g) \in \mathcal{S}_{\omega}(\mathbb{R}^{2d})$. For $ \xi=( \xi_{1}, \xi_{2}) \in \mathbb{R}^{2d}$, we set $ \tilde{\xi}=( \xi_{2},- \xi_{1})$. By \cite[Lemma 2.2]{CG}
\begin{eqnarray}
\nonumber
&& \| \W( \gamma, \psi) \|_{\textbf{M}^{1,p}_{m_{\lambda,2}}}=  
\| V_{\Psi} \W( \gamma, \psi) \|_{L^{1,p}_{m_{\lambda,2}}}\\
\nonumber
&&= \Big( \int_{\mathbb{R}^{2d}} \Big( \int_{\mathbb{R}^{2d}} 
\Big| V_{g} \psi \Big(z+ \frac{\tilde{\xi}}{2} \Big) V_{g} \gamma \Big(z- \frac{\tilde{\xi}}{2} \Big) \Big| \, dz \Big)^{p} m_{\lambda,2}^{p} (\xi)\, d \xi \Big)^{\frac{1}{p}}. \nonumber
\end{eqnarray}

By the change of variables $z+ \frac{\tilde{\xi}}{2}=\tilde{z}$ and \cite[formula (3.12)]{BJO2}
we obtain (cf. also \cite[Prop. 2.5]{CG}):
\begin{eqnarray}
\nonumber
\| \W( \gamma, \psi ) \|_{\textbf{M}^{1,p}_{m_{\lambda,2}}} &=& \biggl( \int_{\mathbb{R}^{2d}} \biggl( \int_{\mathbb{R}^{2d}} | V_{g} \psi ( \tilde{z})| | V_{g} \gamma (\tilde{z}- \tilde{\xi})| \, d \tilde{z} \biggl)^{p} m_{\lambda,2}^{p} (\xi)\, d \xi \biggl)^{\frac{1}{p}}.\\
\nonumber
& =& \biggl( \int_{\mathbb{R}^{2d}}(|V_{g} \psi(\tilde z)|*|V_{g} \gamma(-\tilde z)|)^{p}(\tilde{\xi}) \,m_{\lambda,2}^{p}( \tilde{\xi}) \, d \tilde{\xi} \biggl)^{\frac{1}{p}}\\
\label{art1c}
&& 
 \leq 
 \| V_{g} \psi\|_{L^{1}_{v_{\lambda}}}\| V_{g}\gamma\|_{L^{p}_{m_{\lambda}}} = 
 \| \psi \|_{\textbf{M}^{1}_{v_{\lambda}}} \| \gamma \|_{\textbf{M}^{p}_{m_{\lambda}}}.
\end{eqnarray}
Therefore
$ \W(\gamma,\psi) \in \textbf{M}^{1}_{m_{\lambda,2}}( \mathbb{R}^{2d})$ and hence, from Proposition \ref{24Gb} (with $p=t=r= + \infty$, $ q=s=t'=1$, $ \lambda=0$ and $ \mu=- \lambda$), we have that $\textbf{M}^{\infty}_{m_{- \lambda,2}}* \textbf{M}^{1}_{m_{\lambda,2}} \subseteq \textbf{M}^{\infty, 1}$, so that $ a^{w}=a*\W( \gamma, \psi) \in \textbf{M}^{\infty,1}$ and by \eqref{in2} with $ \mu=0$
\begin{equation}
\nonumber
\| L^{a}_{\psi, \gamma} \|_{op} \leq \| a^{w} \|_{\textbf{M}^{\infty,1}}.
\end{equation}
From \eqref{2632} and \eqref{art1c} we finally have
\begin{eqnarray}
\nonumber
\| L^{a}_{\psi, \gamma} \|_{op} && \leq  \|  a* \W( \gamma, \psi ) \|_{\textbf{M}^{\infty,1}} \leq \| a \|_{\textbf{M}^{\infty}_{m_{- \lambda, 2}}} \| \W( \gamma, \psi) \|_{\textbf{M}^{1}_{m_{\lambda, 2}}} \\
\nonumber
&& \leq \| a \|_{\textbf{M}^{\infty}_{m_{- \lambda, 2}}} \| \psi \|_{\textbf{M}^{1}_{v_{\lambda}}} \| \gamma \|_{\textbf{M}^{p}_{m_{\lambda}}}.
\end{eqnarray}
\end{proof}

A boundedness result analogous to that of Theorem~\ref{lim1} is proved, with different techniques, in
\cite{P} under further restrictions on the symbol $a(x,\xi)$ and without estimates on the norm
of  $L^a_{\psi,\gamma}$.

Set now
$$ 
\textbf{M}_{m_{\lambda}}^{0,1}(\mathbb{R}^{d})= \{f \in \textbf{M}^{\infty,1}_{m_{\lambda}}(\mathbb{R}^{d}): \lim_{|x| \to \infty} \| V_{g}f(x,.) \|_{L^{1}_{m_{\lambda}}} e^{\lambda \omega(x)}=0 \}
$$
and prove the following compactness result (cf. also \cite[Prop. 2.3]{BGHO} and
\cite[Thm. 3.22]{FGP}):
\begin{Th}
\label{gh1}
If $a^{w} \in \textbf{M}_{m_{\lambda}}^{0,1}( \mathbb{R}^{2d})$ for some $ \lambda\geq 0$, then $L^{a^{w}}$ is a compact mapping of $ \textbf{M}^{p,q}_{m_{\lambda}}(\mathbb{R}^{d})$ into itself, for $ 1 \leq p,q < + \infty$.
\end{Th}

\begin{proof}
The operator $L^{a^{w}}$ maps $ \textbf{M}^{p,q}_{m_{\lambda}}(\mathbb{R}^{d})$ into itself by \eqref{in2}. To prove that $L^{a^{w}}$ is compact we first assume $ a^{w} \in \mathcal{S}_{\omega}( \mathbb{R}^{2d})$. From \eqref{weyl}
\begin{eqnarray}
\label{81}
\nonumber
L^{a^{w}}f(y) =&&  \frac{1}{(2 \pi)^{d}} \int_{\mathbb{R}^{2d}} \hat{a}^{w}(\xi,u) e^{-i \xi \cdot u} e^{i \xi \cdot (y+u)} f(y+u) \, du \, d \xi \\ \nonumber
=&&  \frac{1}{(2 \pi)^{d}} \int_{\mathbb{R}^{2d}} \hat{a}^{w}(\xi,x-y) e^{i \xi \cdot y} f(x) \, dx \, d \xi\\ 
=&& \int_{\mathbb{R}^{d}} k(x,y) f(x) \, dx, 
\end{eqnarray}
with kernel $k(x, y)= \frac{1}{(2 \pi)^{d}}\int_{\mathbb{R}^{d}} \hat{a}^{w}(\xi,x-y) 
e^{i \xi \cdot y} d \xi$.
Note that $k(x,y) \in \mathcal{S}_{\omega }( \mathbb{R}^{2d})$ because it is the inverse Fourier transform (with respect to the first variable) of the traslation (with respect to the second variable) of $ \hat{a}^{w} \in \mathcal{S}_{\omega}( \mathbb{R}^{2d})$.
\\ Now, let $ \phi \in \mathcal{S}_{\omega}(\mathbb{R}^{d})$ and $ \alpha_{0}, \beta_{0}>0$ such that $\{ \phi_{jl} \}_{j,l \in \mathbb{Z}^{d}}= \{  M_{\beta_{0}l}T_{\alpha_{0}j} \phi \}_{j,l \in \mathbb{Z}^{d}} $ is a tight Gabor frame for $L^{2}(\mathbb{R}^{d})$ (see \cite[Def. 5.1.1]{G} for the definition). Then $ \{ \Phi_{jlmn} \}_{j,l,m,n \in \mathbb{Z}^{d}}=\{ \phi_{jl}(x) \phi_{mn}(y) \}_{j,l,m,n \in \mathbb{Z}^{d}}$ is a tight Gabor frame for $L^{2}( \mathbb{R}^{2d})$. Since $ k \in \mathcal{S}_{\omega}(\mathbb{R}^{2d})$ we have that $ \langle k, \Phi_{jlmn} \rangle=V_{\phi}k( \alpha_{0}j, \alpha_{0}m, \beta_{0}l, \beta_{0} n) \in \ell^{1}$ and (see \cite[Lemma 3.15]{BJO2})
$$ k= \sum_{j,l,m,n \in \mathbb{Z}^{d}} \langle k, \Phi_{jlmn} \rangle \Phi_{jlmn}.$$ 
Therefore from \eqref{81}
$$ L^{a^{w}}f=\sum_{j,l,m,n \in \mathbb{Z}^{d}} \langle k, \Phi_{jlmn} \rangle \langle \phi_{jl},f  \rangle \phi_{mn},$$
with $ \langle k, \Phi_{jlmn} \rangle \in \ell^{1}$, $ (\phi_{jl})_{j,l \in \mathbb{Z}^{d}}$ equicontinous in $\textbf{M}^{p',q'}_{m_{-\lambda}}= (\textbf{M}^{p,q}_{m_{\lambda}})^{*}$ and $ (\phi_{mn})_{m,n \in  \mathbb{Z}^{d}}$ bounded in $\bigcup_{n \in \mathbb{N}} n \{ f \in \textbf{M}^{\, p,q}_{m_{\lambda}}: \| f \|_{\textbf{M}^{p,q}_{m_{\lambda}}} <1 \},$ so that $L^{a^{w}}$ is a nuclear operator from $\textbf{M}^{p,q}_{m_{\lambda}}$ to $\textbf{M}^{p,q}_{m_{\lambda}}$ (see \cite[\S 17.3]{J}). From \cite[\S 17.3, Cor.~4]{J} we thus have that $L^{a^{w}}$ is compact.
\\Let us finally consider the general case $a \in \textbf{M}^{0,1}_{m_{\lambda}}(\mathbb{R}^{2d})$. By \cite[Prop. 3.9]{BJO2} there exist 
$a_{n} \in \mathcal{S}_{\omega}(\mathbb{R}^{2d})$ converging to $a$ in $ \textbf{M}^{\infty,1}_{m_{\lambda}}$ and hence, by \eqref{in2}
$$ \|  L^{a^w}-  L^{a^w_{n}} \|_{\textbf{M}_{m_{\lambda}}^{\,p,q} \to \textbf{M}_{m_{\lambda}}^{\,p,q}} \leq  \| a- a_{n} \|_{\textbf{M}^{\infty, 1}_{m_{\lambda}}} \to 0.$$
Since the set of compact operators is closed we have that $ L^{a^w}$ is compact on $ \textbf{M}_{m_{\lambda}}^{\,p,q}(\mathbb{R}^{d}).$
\end{proof} 

We have the
following generalization of \cite[Lemma~3.4]{FG} and \cite[Prop.~5.2]{FG2}:
\begin{Lemma}
\label{BD100}
Let $g_{0} \in \mathcal{S}_{\omega}(\mathbb{R}^{d})$ and $ a \in \textbf{M}^{\infty}_{m_{\lambda}}(\mathbb{R}^{d})$, with $ \lambda \geq0$, such that
\newline
\begin{equation}
\label{2M2}
\lim_{|x| \to + \infty} \sup_{| \xi | \leq R} |V_{g_{0}} a(x, \xi)| e^{\lambda \omega(x,\xi)}=0,
\qquad \forall  R>0.
\end{equation}
Then $a*H \in \textbf{M}_{m_{\lambda}}^{0,1}(\mathbb{R}^{d})$ for any $ H \in \mathcal{S}_{\omega}(\mathbb{R}^{d})$.
\end{Lemma}
\begin{proof}
The case $\lambda=0$ has been proved in \cite[Lemma~3.4]{FG}. Let $\lambda>0$.
Since $g_{0} \in \mathcal{S}_{\omega}(\mathbb{R}^{d})$ and  
$H \in \mathcal{S}_{\omega}(\mathbb{R}^{d})$, by 
\cite[Thm. 2.7]{GZ2} we have that 
$V_{g_{0}}H \in \mathcal{S}_{\omega}( \mathbb{R}^{2d})$ and hence, for a fixed $ \ell>0$ (to be chosen later depending on $\lambda$),
there exists $c_{\lambda}>0$ such that
\begin{equation}
\nonumber
|V_{g_{0}}H(x, \xi)| \leq c_{\lambda} e^{-3 \ell \lambda \omega(x)} 
e^{-3 \ell \lambda \omega( \xi)}, \qquad \forall x , \xi \in \mathbb{R}^{d}.
\end{equation}
Now, as in the proof of Proposition \ref{24Gb}, for $g=g_{0}*g_{0}$, we have that 
$| V_{g}(a*H)(\cdot, \xi)|= |V_{g_{0}}a( \cdot, \xi)*V_{g_{0}}H( \cdot, \xi)|$. 
Since $ \omega$ is increasing and subadditive we have
\begin{eqnarray*}
&& | V_{g}(a*H)(x, \xi)|  \leq  \int_{\mathbb{R}^{d}} | V_{g_{0}}a(x-y, \xi )| | V_{g_{0}} H(y, \xi)| dy \\
&& \leq c_{\lambda} e^{-3 \ell \lambda \omega(\xi )}  \!  \int_{\mathbb{R}^{d}} \! | V_{g_{0}}a(x-y, \xi )| e^{-3 \ell \lambda \omega(y)} dy \\
&& = c_{\lambda} e^{-3 \ell \lambda \omega(\xi )}  \!  \int_{\mathbb{R}^{d}} \! | V_{g_{0}}a(x-y, \xi )| e^{-3 \ell \lambda \omega(y)} \, e^{\lambda\omega(x-y,\xi)}e^{-\lambda\omega(x-y,\xi) } dy \\
&& \leq c_{\lambda} e^{-3 \ell\lambda \omega(\xi )} e^{- \lambda \omega(x)} \!  \int_{\mathbb{R}^{d}} \!| V_{g_{0}}a(x-y, \xi )| e^{\lambda \omega(x-y, \xi)}e^{-(3 \ell-1) \lambda \omega(y)} dy.
\end{eqnarray*}
Since $a \in \textbf{M}^{\infty}_{m_{\lambda}}(\mathbb{R}^{d})$  we have that
\begin{eqnarray}
\nonumber
&& e^{\lambda \omega(x)+ 2 \ell\lambda \omega(\xi)} | V_{g}(a*H)(x, \xi)|\\  
\label{ultimo} 
\leq && c_{\lambda} e^{- \ell \lambda  \omega(\xi)}  \int_{\mathbb{R}^{d}} | V_{g}a(x-y, \xi )| e^{\lambda \omega(x-y, \xi)}e^{-(3 \ell-1) \lambda \omega(y)}  dy \\ 
\label{2M1}
 \leq && c_{\lambda} e^{- \ell \lambda \omega(\xi)}  \| a \|_{\textbf{M}^{\infty}_{m_{\lambda}}} \int_{\mathbb{R}^{d}} e^{-(3 \ell-1) \lambda \omega(y)} dy< + \infty,
\end{eqnarray}
if $\ell >\frac{1}{3}+ \frac{d}{3 B \lambda}$, where B is the constant of condition ($ \gamma$) in Definition \ref{peso}. Since $ \lim_{|\xi| \to + \infty} \omega( \xi)=+ \infty$, from \eqref{2M1} we have that for all $ \varepsilon>0$ there exists $R_{1}>0$ such that
\begin{equation}
\label{82}
 e^{\lambda \omega(x)+ 2 \ell\lambda \omega(\xi)}  | V_{g}(a*H)(x, \xi)| < \varepsilon, \quad \forall x, \xi \in \mathbb{R}^{d}, \quad  | \xi| \geq R_{1}.
\end{equation}
We now choose $\delta>0$ small enough so that
\begin{equation}
\label{spa1}
\delta \biggl(1 + c_{\lambda} \int_{\mathbb{R}^{d}} e^{-(3\ell-1) \lambda \omega(y)}\biggl)  dy \leq \varepsilon.
\end{equation}
From the hypothesis \eqref{2M2} we can choose $R_{2}>0$ sufficiently large so that
\begin{equation}
\label{spa3}
 \sup_{| \xi| \leq R_{1}} | V_{g_{0}}a(x, \xi)| e^{\lambda \omega(x, \xi)}< \delta, \quad  |x| \geq R_{2},
\end{equation}
\begin{equation}
\label{spa2}
\int_{|y| > R_{2}} e^{-(3 \ell-1) \lambda \omega(y)} \, dy < \frac{\delta}{c_{\lambda}
e^{- \ell \lambda \omega(\xi) }\| a \|_{\textbf{M}_{m_{\lambda}}^{\infty}}}, \qquad | \xi| \leq R_{1}.
\end{equation}
Therefore for $|x| \geq 2 R_{2}$, $| y | \leq R_{2}$ (so that $ |x-y| \geq R_{2}$) and $ | \xi| \leq R_{1}$, by \eqref{ultimo}, \eqref{2M1},  \eqref{spa2}, \eqref{spa3} and \eqref{spa1}:
\begin{eqnarray*}
&& e^{\lambda \omega(x)+ 2 \ell\lambda \omega(\xi)}  | V_{g}(a*H)(x, \xi)| \\ 
 \leq && c_{\lambda} e^{ - \ell \lambda \omega(\xi)}  \| a \|_{\textbf{M}_{m_{\lambda}}^{\infty}} \int_{|y| > R_{2}} e^{-(3 \ell-1) \lambda \omega(y)}  dy \\
&& + c_{\lambda} e^{-\ell \lambda \omega(\xi )} \int_{|y| \leq R_{2}} | V_{g_{0}}a(x-y, \xi )| e^{ \lambda \omega(x-y, \xi)} e^{- (3 \ell-1) \lambda \omega(y)}  dy \\
 < && \delta + c_{\lambda} \delta \int_{\mathbb{R}^{d}} e^{-(3 \ell-1) \lambda \omega(y)}  dy \leq \varepsilon.
\end{eqnarray*}
The above estimate, together with \eqref{82}, gives
$$ 
e^{\lambda \omega(x)} \int_{\mathbb{R}^{d}} | V_{g}(a*H)(x, \xi)| e^{\lambda \omega(\xi)}  
d \xi \leq \varepsilon \int_{\mathbb{R}^{d}} e^{- (2 \ell-1) \lambda \omega(\xi)} 
d \xi, \qquad |x| \geq 2 R_{2}.
$$
Choosing now $ \ell >\frac{1}{2}+ \frac{d}{2 B \lambda}>\frac{1}{3}+ \frac{d}{3 B \lambda}$ so that $e^{-(2 \ell-1) \lambda \omega(\xi)} \in L^{1}(\mathbb{R}^{d})$, we finally obtain
$$
 \lim_{|x| \to \infty}  e^{\lambda \omega(x)}  \| V_{g}(a*H)(x,.) \|_{L^{1}_{m_{\lambda}}}=0.
 $$
\end{proof}

\begin{Th}
\label{cond1}
Let $\psi, \gamma \in \mathcal{S}_{\omega}(\mathbb{R}^{d})$, $g_{0} \in \mathcal{S}_{\omega}(\mathbb{R}^{2d})$ and $a \in \textbf{M}^{\infty}_{m_{\lambda}}(\mathbb{R}^{2d})$  satisfying \eqref{2M2}, for some $\lambda\geq0$. Then  $L^{a}_{\psi, \gamma}: \textbf{M}_{m_{\lambda}}^{p,q}(\mathbb{R}^{d})\to\textbf{M}_{m_{\lambda}}^{p,q}(\mathbb{R}^{d})$
 is  compact, for $1\leq p,q<+\infty$.
\end{Th}

\begin{proof}
Set $H:= W( \gamma, \psi) \in \mathcal{S}_{\omega}( \mathbb{R}^{2d})$. Since $a \in \textbf{M}^{\infty}_{m_{\lambda}}(\mathbb{R}^{2d})$, by Lemma \ref{BD100} we have that $ a^{w}= a* H \in \textbf{M}_{m_{\lambda}}^{0,1}(\mathbb{R}^{2d})$ 
and hence $L^{a}_{\psi, \gamma}=L^{a^{w}}$ is compact by Theorem~\ref{gh1}.
\end{proof}

\textbf{Acknowledgements}
The authors are grateful to Proff. C.~Fern\'andez, A.~Galbis and D.~Jornet for  helpful discussions. The first author is member of the GNAMPA-INdAM.

\vspace*{3mm}
\begin{center}
Dipartimento di Matematica e Informatica \\Universit\`a di Ferrara,
Via Ma\-chia\-vel\-li n.~30,
I-44121 Ferrara,
Italy
\\
\email{chiara.boiti@unife.it}.
\end{center}

\vspace*{2mm}

\begin{center}
Dipartimento di Matematica \\ Politecnico di Milano,
Via Bonardi n.~9,
20133 Milano,
Italy,
\\
\email{antonino.demartino@polimi.it}.
\end{center}

\end{document}